\documentclass[a4j,12pt,draft]{article}
\usepackage{amsmath}
\usepackage{amssymb}
\usepackage{amsthm}
\usepackage{amscd}
\newtheorem{theorem}{Theorem}[section]
\newtheorem{lemma}{Lemma}[section]
\theoremstyle{definition}
\newtheorem{definition}{Definition}[section]
\theoremstyle{remark}
\newtheorem{remark}{Remark}[section]
\theoremstyle{proposition}
\newtheorem{proposition}{Proposition}[section]
\numberwithin{equation}{section}

\theoremstyle{corollary}
\newtheorem{corollary}{Corollary}[section]

\begin{document}

\title{The Dixmier-Douady class in the Simplicial\\
 de Rham Complex}
\author{ Naoya Suzuki}
\date{}
\maketitle
\begin{abstract}
On the basis of A. L. Carey, D. Crowley, M. K. Murray's work, we exhibit a cocycle in the simplicial de Rham complex
which represents the Dixmier-Douady class.
\end{abstract}

\section{Introduction}

In \cite[Carey,Crowley,Murray]{Car},  they proved that when  a Lie group $G$ admits a central extension $1 \rightarrow U(1) \rightarrow \widehat{G} \rightarrow G \rightarrow 1$, there exists a  characteristic class of principal $G$-bundle $\pi :Y \rightarrow M $ which belongs to a cohomology group $H^2 (M , \underline{U(1)} ) \cong H^3 (M , \mathbb{Z} )$. Here $\underline{U(1)}$ stands for a sheaf of continuous $U(1)$-valued functions on $M$.
 This class is called a Dixmier-Douady class associated to the
central extension $\widehat{G} \rightarrow G$. 

On the other hand,  we have a simplicial manifold $\{ NG(*) \}$ for any Lie group $G$. It is a sequence of manifolds $\{ NG(p) =G^p \}_{ p=0,1, \cdots }$ together
with face maps ${\varepsilon}_{i} : NG(p) \rightarrow NG(p-1)  $ for $ i= 0, \cdots , p $ satisfying relations the ${\varepsilon}_{i}{\varepsilon}_{j} ={\varepsilon}_{j-1}{\varepsilon}_{i}$ for $i<j$.(The standard definition also involves 
degeneracy maps but we do not need them here.)
Then the $n$-th cohomology group of classifying space $BG$ is isomorphic to the total cohomology 
of a double complex $ \{ {\Omega}^{q} (NG(p)) \}_{p+q=n} $. See  \cite{Bot2} \cite{Dup2} \cite{Mos}  for details.
 
In this paper we will exhibit a cocycle on ${\Omega}^{*} (NG(*)) $ which represents the
Dixmier-Douady class due to Carey, Crowley, Murray. Such a cocycle is also studied in a general setting  by
K. Behrend, J.-L. Tu, P. Xu and C. Laurent-Gengoux \cite{Beh} \cite{Beh2} \cite{Tu} \cite{Tu2}, and G. Ginot, M. Sti\'{e}non \cite{Gin}
but our construction of the cocycle is different from theirs, and the proof is more simple. 
Stevenson \cite{Ste} also exhibited a cocycle which represents the
Dixmier-Douady class  in singular 
cohomology group instead of the de Rham cohomology. 
As a consequence of our result, we can show that if $G$ is given a discrete topology, the Dixmier-Douady class  in $H^3 (BG^{\delta} , \mathbb{R} )$ is $0$. Furthermore, we can exhibit
the ``Chern-Simons form" of Dixmier-Douady class on ${\Omega}^{*} (N \bar{G}(*)) $. Here $N \bar{G}$ is a simplicial manifold
which plays the role of universal bundle.

The outline is as follows. In section 2, we briefly recall the notion of simplicial manifold $NG$
and construct a cocycle in ${\Omega}^{*} (NG(*)) $.
In section 3, we recall the definition of a Dixmier-Douady class and prove the main theorem. In section 4,
we give the Chern-Simons form of the Dixmier-Douady class.

\section{  Cocycle on the double complex}
In this section first we recall the relation between the simplicial manifold $NG$ and the classifying space $BG$, then we construct the cocycle on $\Omega ^{*,*}(NG) $.

\subsection{The double complex on simplicial manifold}

For any Lie group $G$ , we define simplicial manifolds $NG$, $N \bar{G}$ and a simplicial $G$-bundle  $\gamma : N \bar{G} \rightarrow NG$
as follows:\\
\par
$NG(p)  = \overbrace{G \times \cdots \times G }^{p-times}  \ni (g_1 , \cdots , g_p ) :$  \\
face operators \enspace ${\varepsilon}_{i} : NG(p) \rightarrow NG(p-1)  $
$$
{\varepsilon}_{i}(g_1 , \cdots , g_p )=\begin{cases}
(g_2 , \cdots , g_p )  &  i=0 \\
(g_1 , \cdots ,g_i g_{i+1} , \cdots , g_p )  &  i=1 , \cdots , p-1 \\
(g_1 , \cdots , g_{p-1} )  &  i=p
\end{cases}
$$

\par
\medskip
$N \bar{G} (p) = \overbrace{ G \times \cdots \times G }^{p+1 - times} \ni (h_1 , \cdots , h_{p+1} ) :$ \\
face operators \enspace $ \bar{\varepsilon}_{i} : N \bar{G}(p) \rightarrow N \bar{G}(p-1)  $ 
$$ \bar{{\varepsilon}} _{i} (h_1 , \cdots , h_{p+1} ) = (h_1 , \cdots , h_{i} , h_{i+2}, \cdots , h_{p+1})  \qquad i=0,1, \cdots ,p $$

\par
\medskip

And we define $\gamma : N \bar{G} \rightarrow NG $ as $ \gamma (h_1 , \cdots , h_{p+1} ) = (h_1 {h^{-1} _2} , \cdots , h_{p} {h^{-1} _{p+1}} )$.\\
\par
To any simplicial manifold $X = \{ X_* \}$, we can associate a topological space $\parallel X \parallel $ 
called the fat realization. 
Since any $G$-bundle $\pi : E \rightarrow M$ can be realized as the pull-back of the fat realization of $\gamma $,
$\parallel \gamma \parallel$ is an universal bundle $EG \rightarrow BG$ \cite{Seg}. \\

Now we construct a double complex associated to a simplicial manifold.

\begin{definition}
For any simplicial manifold $ \{ X_* \}$ with face operators $\{ {\varepsilon}_* \}$, we define double complex as follows:
$${\Omega}^{p,q} (X) \buildrel \mathrm{def} \over = {\Omega}^{q} (X_p) $$
Derivatives are:
$$ d' := \sum _{i=0} ^{p+1} (-1)^{i} {\varepsilon}_{i} ^{*}  , \qquad  d'' := {\rm derivatives \enspace on \enspace } {X_p } \times (-1)^{p} $$
\end{definition}
\hspace{30em} $ \Box $ \\
For $NG$ and $N \bar{G} $ the following holds(\cite{Bot2}\cite{Dup2}\cite{Mos}).

\begin{theorem}
{There exist ring isomorphisms }

$$ H({\Omega}^{*} (NG))  \cong  H^{*} (BG ), \qquad  H({\Omega}^{*} (N \bar{G})) \cong H^{*} (EG ) $$

{ Here} ${\Omega}^{*} (NG)$ { and } ${\Omega}^{*} (N \bar{G})$ { mean the total complexes}.
\end{theorem} 

\hspace{30em} $ \Box $ \\

For a principal $G$-bundle $Y \rightarrow M$ and an open covering $\{ U_{\alpha } \}$ of $M$, the transition functions 
$ (g_{\alpha _0 \alpha _1},g_{\alpha _1 \alpha _2}, \cdots ,
g_{\alpha _{p-1} \alpha _p}):U_{\alpha _0 \alpha _1 \cdots \alpha _p} \rightarrow NG(p)$ induce the cohomology map 
$H^* (NG) \rightarrow H^* _{\check{C}ech-de Rham} (M)$. The elements in the image are the characteristic class of $Y$ \cite{Mos}.

\subsection{Construction of the cocycle }
Let $\rho : \widehat{G} \rightarrow G $ be a central extension of a Lie group $G$ and we recognize it as a $U(1)$-bundle. Using the face operators $\{ {\varepsilon}_{i} \}: NG(2) \rightarrow NG(1) =G $, 
we can construct the $U(1)$-bundle over $NG(2)= G \times G $ as $\delta \widehat{G} :=  {\varepsilon _0}^* \widehat{G} \otimes ({{\varepsilon} _1}^* \widehat{G})^{{\otimes}-1} \otimes
{{\varepsilon} _2}^* \widehat{G} $.
Here we define the tensor product $S \otimes T$ of $U(1)$-bundles $S$ and $T$ over $M$ as
$$ S \otimes T := \bigcup _{x \in M} (S_x \times T_x) / (s,t) \sim (su,tu^{-1}), (u \in U(1))$$

\par
\begin{lemma}
 $\delta \widehat{G} \rightarrow G \times G $ { is a trivial bundle.}
\end{lemma}

\begin{proof}
We can construct a bundle isomorphism 
$f:{{\varepsilon}_0} ^* \widehat{G} \otimes {{\varepsilon}_2} ^* \widehat{G} \rightarrow
{{\varepsilon}_1} ^* \widehat{G} $ as follows. First we define
$ f$ to be the map sending $[((g_1,g_2),\hat{g}_2 ),((g_1,g_2),\hat{g}_1 )] $s.t. $\rho ( \hat{g}_2 )=g_2,
\rho ( \hat{g}_1 )=g_1 $ to $((g_1,g_2),\hat{g}_1 \hat{g}_2) $. Then we have the inverse
$ f^{-1} $ that sends $ ((g_1,g_2), \hat{g})$ s.t. $\rho ( \hat{g} )=g_1g_2$ to $[((g_1,g_2),\hat{g}_2 ),((g_1,g_2),\hat{g} \hat{g}^{-1} _2 )]$ s.t. 
$\rho ( \hat{g}_2 )=g_2 $
\end{proof}

\bigskip

For any connection $\theta$ on $\widehat{G}$, there is the induced connection $\delta \theta $ on $\delta \widehat{G} $ \cite[Brylinski]{Bry}.

\begin{proposition}
{Let} $c_1 (\theta )$ denote the 2-form on $G$ which hits $\left( \frac{-1}{2 \pi i} \right) d \theta \in \Omega ^2 (\widehat{G})$ by $\rho ^{*}$, and $\hat{s}$ any global section of
$\delta \widehat{G} $. Then the following equation holds.
$$ ({\varepsilon}_{0} ^* - {\varepsilon}_{1} ^* + {\varepsilon}_{2} ^* ) c_1 (\theta ) = \left( \frac{-1}{2 \pi i} \right)d(\hat{s}^{*} (\delta \theta )) \enspace \in \Omega ^2 (NG(2)) .$$ 
\end{proposition}
\begin{proof}
Choose an open cover $\mathcal{V} = \{ V_{\lambda} \}_{\lambda \in \Lambda}$ of $G$ such that there exist local sections $ {\eta}_{\lambda} : V_{\lambda} \rightarrow \widehat{G}$ of $\rho$.
Then  $ \{ {\varepsilon}_{0} ^{-1} (V_{\lambda} ) \cap {\varepsilon}_{1} ^{-1} (V_{\lambda '} ) \cap {\varepsilon}_{2} ^{-1} (V_{\lambda ''}) \}_{\lambda , {\lambda}' , {\lambda}'' \in \Lambda}  $ is an open cover of $G \times G$ and there are the induced local
sections ${\varepsilon}_{0} ^* \eta _{\lambda} \otimes ({\varepsilon}_{1} ^* \eta _{\lambda '})^{\otimes -1} \otimes {\varepsilon}_{2} ^* \eta _{\lambda ''} $ on that covering.

If we pull back $\delta \theta $ by these sections, the induced form on ${\varepsilon}_{0} ^{-1} (V_{\lambda} ) \cap {\varepsilon}_{1} ^{-1} (V_{\lambda '} ) \cap {\varepsilon}_{2} ^{-1} (V_{\lambda ''})$ is ${\varepsilon}_{0} ^* (\eta _{\lambda} ^* \theta )- {\varepsilon}_{1} ^* (\eta _{\lambda '} ^* \theta )+ {\varepsilon}_{2} ^* ( \eta _{\lambda ''} ^* \theta )$.
 We restrict $({\varepsilon}_{0} ^* - {\varepsilon}_{1} ^* + {\varepsilon}_{2} ^* ) c_1 (\theta )$ on ${\varepsilon}_{0} ^{-1} (V_{\lambda} ) \cap {\varepsilon}_{1} ^{-1} (V_{\lambda '} ) \cap {\varepsilon}_{2} ^{-1} (V_{\lambda ''})$ then it is equal to 
 $ \left( \frac{-1}{2 \pi i} \right)d({\varepsilon}_{0} ^* (\eta _{\lambda} ^* \theta )- {\varepsilon}_{1} ^* (\eta _{\lambda '} ^* \theta )+ {\varepsilon}_{2} ^* ( \eta _{\lambda ''} ^* \theta ))$, because $c_1 (\theta ) = \sum  \left( \frac{-1}{2 \pi i} \right) d( \eta _{\lambda} ^* \theta )$.\\
Also $d({\varepsilon}_{0} ^* (\eta _{\lambda} ^* \theta )- {\varepsilon}_{1} ^* (\eta _{\lambda '} ^* \theta )+ {\varepsilon}_{2} ^* ( \eta _{\lambda ''} ^* \theta )) = d(\hat{s}^{*} (\delta \theta ))|_{{\varepsilon}_{0} ^{-1} (V_{\lambda} ) \cap {\varepsilon}_{1} ^{-1} (V_{\lambda '} ) \cap {\varepsilon}_{2} ^{-1} (V_{\lambda ''})}$
since $\delta \theta $ is a connection form. This completes the proof.
\end{proof}

\begin{proposition}
{ For the face operators } $  \{{\varepsilon}_{i} \}_{i=0,1,2,3}: NG(3) \rightarrow NG(2)$,\\
$$({\varepsilon}_{0} ^* - {\varepsilon}_{1} ^* + {\varepsilon}_{2} ^* - {\varepsilon}_{3} ^*)(\hat{s}^{*} (\delta \theta ))=0.$$
\end{proposition}
\begin{proof}
We consider the $U(1)$-bundle $\delta (\delta \widehat{G})$ over $NG(3)=G \times G \times G$ and the induced connection $\delta (\delta \theta )$ on it.
Composing $  \{{\varepsilon}_{i} \}: NG(3) \rightarrow NG(2)$ and $  \{{\varepsilon}_{i} \}: NG(2) \rightarrow G$, we define the maps $\{ r_i\} _{i=0,1, \cdots ,5} :NG(3) \rightarrow G$ as follows.
$$ r_0 = {\varepsilon}_{0} \circ {\varepsilon}_{1}= {\varepsilon}_{0} \circ {\varepsilon}_{0} \thinspace , \quad r_1 = {\varepsilon}_{0} \circ {\varepsilon}_{2}= {\varepsilon}_{1} \circ {\varepsilon}_{0} \thinspace
, \quad r_2 = {\varepsilon}_{0} \circ {\varepsilon}_{3}= {\varepsilon}_{2} \circ {\varepsilon}_{0}$$
$$ r_3 = {\varepsilon}_{1} \circ {\varepsilon}_{2}= {\varepsilon}_{1} \circ {\varepsilon}_{1} \thinspace , \quad r_4 = {\varepsilon}_{1} \circ {\varepsilon}_{3}= {\varepsilon}_{2} \circ {\varepsilon}_{1} \thinspace
, \quad r_5 = {\varepsilon}_{2} \circ {\varepsilon}_{3}= {\varepsilon}_{2} \circ {\varepsilon}_{2}$$
Then $\{ \bigcap r^{-1} _i (V_{{\lambda} ^{(i)}} ) \}$ is a covering of $NG(3)$. Since each $ \bigcap r^{-1} _i (V_{{\lambda} ^{(i)}} ) $
is equal to
$${\varepsilon}_{0} ^{-1}({\varepsilon}_{0} ^{-1} (V_{\lambda} ) \cap {\varepsilon}_{1} ^{-1} (V_{\lambda '} ) \cap {\varepsilon}_{2} ^{-1} (V_{\lambda ''}))
 \cap {\varepsilon}_{1} ^{-1}({\varepsilon}_{0} ^{-1} (V_{\lambda} ) \cap {\varepsilon}_{1} ^{-1} (V_{\lambda ^{(3)}} ) \cap {\varepsilon}_{2} ^{-1} (V_{\lambda ^{(4)}}))$$
$$\cap {\varepsilon}_{2} ^{-1}({\varepsilon}_{0} ^{-1} (V_{\lambda '} ) \cap {\varepsilon}_{1} ^{-1} (V_{\lambda ^{(3)}} ) \cap {\varepsilon}_{2} ^{-1} (V_{\lambda ^{(5)}}))
 \cap {\varepsilon}_{3} ^{-1}({\varepsilon}_{0} ^{-1} (V_{\lambda ''} ) \cap {\varepsilon}_{1} ^{-1} (V_{\lambda ^{(4)}} ) \cap {\varepsilon}_{2} ^{-1} (V_{\lambda ^{(5)}}))$$ 
there are the following induced local sections on that.
$${\varepsilon}_{0} ^*({\varepsilon}_{0} ^* \eta _{\lambda} \otimes ({\varepsilon}_{1} ^* \eta _{\lambda '})^{\otimes -1} \otimes {\varepsilon}_{2} ^* \eta _{\lambda ''}) \otimes {\varepsilon}_{1} ^*({\varepsilon}_{0} ^* \eta _{\lambda} \otimes ({\varepsilon}_{1} ^* \eta _{\lambda ^{(3)}})^{\otimes -1} \otimes {\varepsilon}_{2} ^* \eta _{\lambda ^{(4)}})^{\otimes -1}$$
$$\otimes {\varepsilon}_{2} ^*({\varepsilon}_{0} ^* \eta _{\lambda '} \otimes ({\varepsilon}_{1} ^* \eta _{\lambda ^{(3)}})^{\otimes -1} \otimes {\varepsilon}_{2} ^* \eta _{\lambda ^{(5)}}) \otimes {\varepsilon}_{3} ^*({\varepsilon}_{0} ^* \eta _{\lambda ''} \otimes ({\varepsilon}_{1} ^* \eta _{\lambda ^{(4)}})^{\otimes -1} \otimes {\varepsilon}_{2} ^* \eta _{\lambda ^{(5)}})^{\otimes -1}.$$
From direct computations we can check that the pull-back of $\delta (\delta \theta )$
by this section is equal to 0. This means $\delta (\delta \theta )$ is the Maurer-Cartan connection. Hence if we pull back $\delta (\delta \theta )$ by the induced section ${\varepsilon}_{0} ^* \hat{s} \otimes ({\varepsilon}_{1} ^* \hat{s})^{\otimes -1} \otimes {\varepsilon}_{2} ^* \hat{s} \otimes ({\varepsilon}_{3} ^* \hat{s})^{\otimes -1}$, it is also equal to 0 and this pull-back is nothing but $({\varepsilon}_{0} ^* - {\varepsilon}_{1} ^* + {\varepsilon}_{2} ^* - {\varepsilon}_{3} ^*)(\hat{s}^{*} (\delta \theta ))$.
\end{proof}
\par

\rm{The propositions above give the cocycle} $c_1 (\theta ) - \left( \frac{-1}{2 \pi i} \right)\hat{s}^{*} (\delta \theta ) \in \Omega ^3 (NG)$ below.

$$
\begin{CD}
0 \\
@AA{d}A \\
c_1 ( \theta ) \in {\Omega}^{2} (G )@>{{\varepsilon}_{0} ^* - {\varepsilon}_{1} ^* + {\varepsilon}_{2} ^* }>>{\Omega}^{2} (G \times G)\\
@.@AA{-d}A\\
@.-\left( \frac{-1}{2 \pi i} \right) \hat{s}^{*} (\delta \theta ) \in {\Omega}^{1} (G \times G)@>{{\varepsilon}_{0} ^* - {\varepsilon}_{1} ^* + {\varepsilon}_{2} ^* - {\varepsilon}_{3} ^*}>> 0
\end{CD}
$$

\begin{proposition}
{The cohomology class} $[c_1 (\theta ) - \left( \frac{-1}{2 \pi i} \right)\hat{s}^{*} (\delta \theta )] \in H^3 (\Omega  (NG))$ does not depend on $\theta$.
\end{proposition}
\begin{proof}
Suppose $\theta _0$ and $\theta _1$ are two connections on $\widehat{G}$.Consider the  $U(1)$-bundle
$ \widehat{G} \times [0,1] \rightarrow G \times [0,1] $ and the connection form $t \theta _0 + (1-t) \theta _1 $ on it. Then
we obtain the cocycle $c_1 ( t \theta _0 + (1-t) \theta _1) - \left( \frac{-1}{2 \pi i} \right)\hat{s}^{*} (\delta (t \theta _0 + (1-t) \theta _1) )$ on $ \Omega ^3 (NG \times [0,1])$. Let $i_0 : NG \times \{0 \} \rightarrow NG \times [0,1]$ and $i_1 : NG \times \{1 \} \rightarrow NG \times [0,1]$ be the natural inclusion map. When we identify $ NG \times \{0 \} $ with $NG \times \{1 \}$, $ (i_0 ^{*})^{-1}
 i_1 ^{*} :H({\Omega}^{*} (NG \times \{0 \})) \rightarrow H({\Omega}^{*} (NG \times \{1 \}))$ is the identity map.Hence $ 
[c_1 (\theta _0) - \left( \frac{-1}{2 \pi i} \right)\hat{s}^{*} (\delta \theta _0)]
=[c_1 (\theta _1) - \left( \frac{-1}{2 \pi i} \right)\hat{s}^{*} (\delta \theta _1)]$.
\end{proof}

\section{Dixmier-Douady class on the double complex}

First, we recall the definition of Dixmier-Douady classes, following \cite{Car}.
Let $\pi : Y \rightarrow M $ be a principal $G$-bundle  and  $\{ U_{\alpha } \}$ a Leray covering of $M$.
When $G$ has a central extension $\rho : \widehat{G} \rightarrow G $,
  the transition functions $ g _{\alpha \beta } : U_{\alpha \beta } \rightarrow G$ lift to $ \widehat{G} $. i.e.
there exist continuous maps $\hat{g} _{\alpha \beta } : U_{\alpha \beta } \rightarrow \widehat{G} $ such that
$\rho \circ \hat{g}_{\alpha \beta } = g _{\alpha \beta }$. This is because each $U_{\alpha \beta }$ is contractible so
the pull-back of $\rho $ by $ g _{\alpha \beta }$ has a global section.
Now the $U(1)$-valued functions $c_{\alpha \beta \gamma}$ on $U_{\alpha \beta \gamma}$ are defined
as $c_{\alpha \beta \gamma}:= \hat{g} _{ \beta \gamma }  \hat{g}^{-1} _{\alpha \gamma   } \hat{g} _{\alpha \beta } $. Note
that here they identify $ g_{ \beta \gamma } ^*  \widehat{G} \otimes ( g_{ \alpha \gamma } ^*  \widehat{G})^{\otimes -1}
\otimes g_{\alpha \beta  } ^*  \widehat{G}$ with $U_{\alpha \beta \gamma} \times U(1)$.
Then it is easily seen that $\{ c_{\alpha \beta \gamma} \}$ is a $U(1)$-valued \v{C}ech-cocycle on $M$ and 
hence define a cohomology class in $H^2 (M , \underline{U(1)} ) \cong H^3 (M , \mathbb{Z} )$.
 This class is called the Dixmier-Douady class of $Y$. 
 
Here $G$ can be infinite dimensional, 
but we require $G$ to have a partition of unity so that we can consider a connection form on the $U(1)$-bundle over $G$.
A good example which satisfies such a condition is the loop group of a finite dimensional Lie group \cite{Bry} \cite{Pre}.

Secondly, we fix any trivialization $\delta \widehat{G} \cong \widehat{G} \times U(1)$. Then since $ g_{ \beta \gamma } ^*  \widehat{G} \otimes ( g_{ \alpha \gamma } ^*  \widehat{G})^{\otimes -1} \otimes g_{\alpha \beta  } ^*  \widehat{G}$ is the pull-back of $\delta \widehat{G}$
by $ ( g_{\alpha \beta  } , g_{ \beta \gamma } ) : U_{\alpha \beta \gamma} \rightarrow G \times G $,  there is the induced trivialization $ g_{ \beta \gamma } ^*  \widehat{G} \otimes ( g_{ \alpha \gamma } ^*  \widehat{G})^{\otimes -1}
\otimes g_{\alpha \beta  } ^*  \widehat{G} \cong U_{\alpha \beta \gamma} \times U(1)$. So we have the Dixmier-Douady cocycle
by using this identification.

Now we are ready to state the main theorem.
\begin{definition}
\rm{ For the global section }$ \hat{s} : G \times G \rightarrow  1$, we call  the sum of $c_1 (\theta ) \in \Omega ^2(NG(1)) $ and $ - \left( \frac{-1}{2 \pi i}  \right)\hat{s}^{*} (\delta \theta ) \in \Omega ^1 (NG(2))$ the simplicial Dixmier-Douady cocycle associated to $ \theta $ and the trivialization $\delta \widehat{G} \cong \widehat{G} \times U(1)$.
\end{definition}
\begin{theorem}

{The} simplicial Dixmier-Douady cocycle represents the universal Dixmier-Douady class associated to $\rho$.

\end{theorem}

\begin{proof}
We show that the $ [ C_{2,1} + C_{1,2}]$ below is equal to $ [ \{  \left( \frac{-1}{2 \pi i} \right) d \log c_{\alpha \beta \gamma} \} ] $ as a \v{C}ech-de Rham cohomology class of $M = \bigcup U_{\alpha } $.
\\

$$
\begin{CD}
{C_{2,1} \in \prod {\Omega}^{2} (U_{\alpha \beta })}\\
@AA{-d}A\\
{  \prod {\Omega}^{1} (U_{\alpha \beta })}@>{ \check{\delta} }>> { C_{1,2} \in \prod {\Omega}^{1} (U_{\alpha \beta \gamma}) }\\
\end{CD}
$$

$$C_{2,1} =  \{  ( {g^* _{\alpha \beta }} c_1 (\theta ) ) \}, \qquad
C_{1,2} =  \biggl\{ - \left( \frac{-1}{2 \pi i} \right) ({g _{\alpha \beta }},{g _{\beta \gamma}})^* \hat{s}^{*} (\delta \theta ) \biggr\} $$

Since $ {g^* _{\alpha \beta }} c_1 (\theta ) = {\hat{g}^* _{\alpha \beta }} \rho ^* (c_1 (\theta )) = d \left( \frac{-1}{2 \pi i} \right){\hat{g}^* _{\alpha \beta }} \theta $, we can see $[C_{2,1}+C_{1,2}]=[ \check{\delta} \{ \left( \frac{-1}{2 \pi i} \right){\hat{g}^* _{\alpha \beta }} \theta
\} +C_{1,2}]$.
 By definition 
$ ( \hat{s} \circ ({g _{\alpha \beta }},{g _{\beta \gamma}}))(p) \cdot
c_{\alpha \beta \gamma}(p) = (\hat{g} _{\beta \gamma} \otimes \hat{g}^{\otimes -1} _{\alpha \gamma } \otimes \hat{g} _{\alpha \beta } )(p)
$ for any $p \in U_{\alpha \beta \gamma} $. Hence $ ({g _{\alpha \beta }},{g _{\beta \gamma}})^* \hat{s}^{*} (\delta \theta ) + d \log c_{\alpha \beta \gamma}
=  \check{\delta} \{ \hat{g}^* _{\alpha \beta } \theta \}$.

\end{proof}

\begin{corollary}
{ If the principal} $G$-bundle over $M$ is flat, then its Dixmier-Douady class is $0$ in $H^3 (M , \mathbb{R} )$.
\end{corollary}
\begin{proof}
This is because the cocycle in Theorem 3.1 vanishes when $G$ is given a discrete topology. 
\end{proof}

\begin{corollary}
{If} the first Chern class of $\rho : \widehat{G} \rightarrow G $ is
not $0$ in  $H^2 (G , \mathbb{R} )$, then the corresponding Dixmier-Douady class of the universal $G$-bundle is not $0$.
\end{corollary}
\begin{proof}
In that situation, any differential form $x \in \Omega ^1 (NG(1))$ does not hit $c_1 (\theta ) \in \Omega ^2 (NG(1)) $ by $d:\Omega ^1 (NG(1)) \rightarrow \Omega ^2 (NG(1))$.
\end{proof}

\section{ Chern-Simons form}
As mentioned in section 2.1, $N \bar{G}$ plays the role of the
universal $G$-bundle and $NG$, the classifying space $BG$.
Then, the pull-back of the cocycle in Definition 3.1
 to $\Omega ^{*}(N \bar{G}) $ by $\gamma :N \bar{G} \rightarrow NG$ 
should be a coboundary of a cochain  on $N \bar{G}$.In this section we shall exhibit an explicit form of the 
cochain, which can be called Chern-Simons form for the Dixmier-Douady class. 

Recall $N \bar{G}(1)=G \times G $ and $\gamma : N \bar{G}(1) \rightarrow NG $ is defined as $\gamma (h_1, h_2) = h_1 h^{-1} _2 $.
Then we consider the $U(1)$-bundle $\bar{\delta}_{\gamma}\widehat{G} := {\bar{\varepsilon}_{0}}^{*}\widehat{G} \otimes {\gamma}^* \widehat{G} \otimes ({\bar{\varepsilon}_{1}}^{*}\widehat{G})^{\otimes -1}$ over $G \times G$ and the induced connection $\bar{\delta}_{\gamma} \theta $
on it. We can check $\bar{\delta}_{\gamma} \widehat{G}$ is trivial using the same argument as that in Lemma 2.1, so there is a global section
$\bar{s}_{\gamma} : G \times G \rightarrow \bar{\delta}_{\gamma} \widehat{G} $. 
\begin{theorem}
{If we take} $\bar{s} _{\gamma} =1$, the cochain $ c_1 ( \theta ) - \left( \frac{-1}{2 \pi i} \right) \bar{s}^{*} _{\gamma} (\bar{\delta}_{\gamma} \theta ) \in {\Omega}^{2} (N \bar{G} )$ is a Chern-Simons form of $c_1 (\theta ) - \left( \frac{-1}{2 \pi i} \right)\hat{s}^{*} (\delta \theta ) \in \Omega ^3 (NG)$.
$$
\begin{CD}
0 \\
@AA{d}A \\
c_1 ( \theta ) \in {\Omega}^{2} (G )@>{\bar{\varepsilon}_{0} ^* - \bar{\varepsilon}_{1} ^*  }>>{\Omega}^{2} (N \bar{G}(1))\\
@.@AA{-d}A\\
@.- \left( \frac{-1}{2 \pi i} \right) \bar{s}^{*} _{\gamma} (\bar{\delta}_{\gamma} \theta ) \in {\Omega}^{1} (N \bar{G} (1))@>{\bar{\varepsilon}_{0} ^* - \bar{\varepsilon}_{1} ^* + \bar{\varepsilon}_{2} ^* }>> {\Omega}^{1} (N \bar{G} (2))
\end{CD}
$$
\end{theorem}
\begin{proof}
Repeating the same argument as that in Proposition 2.1,
we can see $(\bar{\varepsilon}_{0} ^* + {\gamma} ^* - \bar{\varepsilon}_{1} ^* )( (c_1 (\theta ))=
\left( \frac{-1}{2 \pi i} \right)d(\bar{s}^{*} _{\gamma} (\bar{\delta}_{\gamma} \theta )) \in \Omega ^2 (N \bar{G}(1))$.
Because $({\varepsilon _0},{\varepsilon _1},{\varepsilon _2}) \circ \gamma = ( \gamma \circ \bar{\varepsilon _0},\gamma \circ \bar{\varepsilon _1},\gamma \circ \bar{\varepsilon _2})$,  $ ( \bar{\varepsilon _0}^* \bar{\delta}_{\gamma}\widehat{G} ) \otimes (\bar{{\varepsilon} _1}^* \bar{\delta}_{\gamma}\widehat{G})^{{\otimes}-1} \otimes
({\bar{\varepsilon} _2}^*  \bar{\delta}_{\gamma} \widehat{G} )$ is $ \gamma ^* (\delta \widehat{G} ) $. Hence
$ ( \bar{\varepsilon}_{0} ^* - \bar{\varepsilon}_{1} ^* + \bar{\varepsilon}_{2} ^* ) \bar{s}^{*} _{\gamma} (\bar{\delta}_{\gamma} \theta )
=  \gamma ^* ( \hat{s}^{*} (\delta \theta ) )$.
\end{proof}
 By restricting the Chern-Simons form on $\Omega ^{*}(N \bar{G}) $ to the edge $\Omega ^{*}(N \bar{G}(0)) $, we obtain the cocycle on $\Omega ^{*}(G) $. So there is the induced map of the cohomology class
$ H^* (BG) \cong H({\Omega}^{*} (N {G})) \rightarrow H^{*-1}(G)$. This map coincides with the transgression map for the universal bundle $EG \rightarrow BG$ in the sense of J. L. Heitsch and H. B. Lawson in \cite{Hei}. Hence as a corollary  of theorem 4.1, we obtain an
alternative proof of the following theorem from \cite{Car} \cite{Ste}.
\begin{theorem}
{The} transgression map of the universal bundle  $EG \rightarrow BG$ maps the Dixmier-Douady class to the first Chern class of $\rho : \widehat{G} \rightarrow G $. 
\end{theorem}
\begin{remark}
Here the meaning of the terminology ``transgression map" is different from those in \cite{Car} \cite{Ste},
but the statement is essentially same.
\end{remark}

\bigskip
\par
{\bf Acknowledgments.} \\
I am indebted to my supervisor, Professor H. Moriyoshi for enlightening discussions and good advice.I would also like to thank K. Gomi for reading earlier drafts and for helpful comments.

Graduate School of Mathematics, Nagoya University, Furo-cho, Chikusa-ku, Nagoya-shi, Aichi-ken, 464-8602, Japan. \\
e-mail: suzuki.naoya@c.mbox.nagoya-u.ac.jp
\end{document}